\documentclass{amsart}
\usepackage{amsmath}
\usepackage{amssymb}
\usepackage{amsthm}
\usepackage{graphicx}
\usepackage{subfigure}
\usepackage{amscd}
\newtheorem{theorem}{Theorem}[section]
\newtheorem{lemma}[theorem]{Lemma}
\newtheorem{prop}[theorem]{Proposition}
\newtheorem{cor}[theorem]{Corollary}

\newcommand{\bH}{{\mathbb H}}

\newcommand{\bR}{{\mathbb R}}
\newcommand{\bZ}{{\mathbb Z}}
\newcommand{\bE}{{\mathbb E}}

\begin{document}
\title{Some 3-manifold groups with the same finite quotients}

\author{John Hempel}

\address{Rice University, Houston, Texas}

\email{hempel@rice.edu}
\begin{abstract}
We give examples of closed, oriented 3-manifolds whose fundamental groups are not isomorphic, but yet have the same sets of finite quotient groups; hence the same profinite completions. We also give examples of compact, oriented 3-manifolds with non-empty boundaries whose fundamental groups though isomorphic have distinct peripheral structures, but yet have the same sets of finite peripheral pair quotients (defined below). The examples are Seifert Fibered Spaces with zero rational Euler number; moreover most of these manifolds give rise to such examples.
\end{abstract}
\maketitle

\section{Introduction} The concepts involved in this paper are only significant in the class of finitely generated, residually finite groups. Finite generation gives that the group has only finitely many normal subgroups of each given finite index -- which is all that is really needed for the group theoretic arguments.  Clearly  compact manifolds have finitely generated fundamental groups.  Thurston's results for Haken 3-manifolds (c.f. \cite{Hemp}) together with Perelman's proof of the geometrization conjecture show that all compact 3-manifolds have residually finite fundamental groups.

Two   groups with the same sets of isomorphism classes of finite quotient groups are known to have isomorphic profinite completions
 \cite{DFPR}. There are fairly simple examples of non- isomorphic (finitely generated, residually finite) groups with the same finite quotients (cf. \cite{Baum}). However these examples seem  unrelated to fundamental groups of low dimensional manifolds. In fact finitely generated Fuchsian groups have been shown to be isomorphic if they have the same finite quotients \cite{BCR}. Moreover for closed $3-$manifolds $M, N$ which admit geometric structures it is shown  in \cite{LoRe} that a homomorphism $\pi_1(M) \to \pi_1(N)$ which induces an isomorphism of the profinite completions  of these groups must be an isomorphism.

We give here examples of compact $3-$manifolds which cannot be distinguished by the finite quotients of their fundamental groups. There are two classes of these examples. The first consist of the closed, oriented Seifert fibered spaces with orientable base and
zero rational Euler number. These are the same as the $3-$manifolds which fiber over $S^1$ with fiber a closed oriented surface $F$ and monodromy $\phi: F \to F$ a periodic homeomorphism. We denote these by

$$
M_\phi = (F \times [0,1])/ (\phi(x),0) =( x,1))
 $$

We show:

\begin{theorem} Let $\phi:F \to F$ be a periodic, orientation preserving homeomorphism of a closed, oriented surface $F$ and let $k$ be relatively prime to $order(\phi)$.  Then $\pi_1(M_\phi)$ and $\pi_1(M_{\phi^k})$ have the same finite quotients; hence the same profinite completitions.
\end{theorem}

In general $\pi_1(M_\phi)$ and $\pi_1(M_{\phi^k})$ will not be isomorphic. This is discussed in section 5. Note that these examples illustrate the necessity of the hypothesis in the above mentioned result of \cite{LoRe}.

Funar \cite{Fun} gives examples of closed 3-manifolds with the geometry $Sol$ which have non-isomorphic fundamental groups with the same finite quotients. Our examples all have the geometry $\bH^2 \times S^1$.

The other class consists of $3-$manifolds with non-empty boundary.
	A  compact  orientable Seifert fibered space with orientable base and non-empty boundary will be a surface bundle over $S^1$ with periodic monodromy \cite{Ja-2}. We use the same notation as above with the understanding that, in this context, $\partial F \ne \emptyset$.  In this case $\pi_1(M_\phi)$ and $\pi_1(M_{\phi^k})$ will be isomorphic, but the corresponding manifolds will be homeomorphic if and only if their peripheral group systems are isomorphic-- which is generally not the case: see section 5.

We will restrict to the case where $\partial F$ has a single component. Thus $\partial M_\phi$ will be connected -- but not conversely. This requirement is because Lemma 2.1 will not apply if $\phi$ were allowed to permute components of $\partial F$.

Unless $F = B^2$, $\partial M_\phi$ will be incompressible in $M_\phi$,  $i_\ast:\pi_1(\partial M_\phi) \to \pi_1(M_\phi)$ will be injective,  and  $i_\ast(\pi_1(\partial M_\phi))$ will be a subgroup well defined up to conjugation. The {\em peripheral group pair}  $(\pi_1(M_\phi), i_\ast(\pi_1(\partial M_\phi)))$ will be well defined up to isomorphism of group pairs. 

For any group pair $(G, H)$ the {\em finite quotients of the pair} will be the group pairs:
$$
\{(G/N, H/N\cap H) \cong (G/N, HN/N); N \text{ a normal subgroup of finite index in } G\}.
$$
We show:

\begin{theorem} Let $\phi:F \to F$ be a periodic, orientation preserving homeomorphism of a compact, oriented surface $F$ which has one boundary component and let $k$ be relatively prime to $order(\phi)$.  Then the peripheral group pairs  $(\pi_1(M_\phi), i_\ast(\pi_1(\partial M_\phi)))$ and $(\pi_1(M_{\phi^k}, i_\ast(\pi_1(\partial M_{\phi^k}))$ have the same finite quotients. 
\end{theorem}
 
 The proofs are given in section 3 after some group theoretic preliminaries in section 2 which, in the case of Theorem 1.1,  show that 
 $$
 \pi_1(M_\phi) \times \mathbb{Z} \cong \pi_1(M_{\phi^k}) \times \mathbb{Z}
 $$
 and then quoting a proposition of Baumslag \cite{Baum} that the the factors $ \pi_1(M_\phi)$ and $ \pi_1(M_{\phi^k})$ have the same finite quotients.
 
 We could as well cite the results of \cite{KwRo} that under the same hypothesis as Theorem 1.1  $M_\phi \times S^1$ and $M_{\phi^k} \times S^1$ are homeomorphic.  We prefer to include the proofs given here because (1) they are elementary, (2) the group theoretic observations should be of independent interest, and (3) they seem to give the most direct accommodation of the reasoning to the case of manifolds with boundary and peripheral fundamental group systems.
 
 \section{Group results} A {\em semi-direct product} of groups $A$ and $B$ is denoted:
 $$
 A \rtimes_\Psi B
 $$
 where $\Psi : B \to Aut(A)$ is a homomorphism describing the action of $B$ on $A$. When $B \cong \mathbb{Z} = <t: \;>$, $\Psi$ is completely determined by $\psi = \Psi(1) \in Aut(A)$, and we will use the abbreviated notation
 $$
 G_\psi = A \rtimes_\psi  \mathbb{Z}.
 $$
 
 \begin{lemma} Let $N$ be a finitely presented group and $\psi \in Aut(N)$ induce  a periodic outer automorphism. Then for any $k \in \mathbb{Z}$ relatively prime to $order(\psi)$ we have
 $$
 G_\psi \times \mathbb{Z} \cong G_{\psi^k} \times \mathbb{Z}.
 $$
 
 Moreover if $B$ is a subgroup of $N$ which is invariant under $\psi$ (so $H_\psi = B \rtimes_\psi \mathbb{Z}$ is a subgroup of $G_\psi$), then there is an isomorphism of pairs:
 $$
 (G_\psi \times \mathbb{Z}, H_\psi \times \mathbb{Z}) \cong (G_{\psi^k} \times \mathbb{Z}, H_{\psi^k} \times \mathbb{Z}).
 $$
 \end{lemma}
 
 \begin{proof} A presentation $<X;R>$ for $N$ extends to the presentation 
 $$
 <X \cup\{t\}: R \cup \bigcup_{x \in X} txt^{-1} = \psi(x)> 
 $$
 for  $G_\psi$. Now
 $$
 G_\psi \times \mathbb{Z} \cong N \rtimes_{\Psi} (\mathbb{Z} \times \mathbb{Z}), \text {and}
 $$
 $$
 H_\psi \times \bZ \cong B \rtimes _\Psi (\bZ \times \bZ)  \cong B \times  (\bZ \times \bZ)
 $$
 where $\Psi(1,0) = \psi$ and $\Psi(0,1) = id$. To extend to a presentation for $G_\psi \times \bZ$ we add a generator $s$ which commutes with everything.
 
Let  $n = order(\psi)$; so that there is some $g\in G$ with $\psi^n(x) = gxg^{-1}$ for all $ x \in G$. We have $1 = an +bk$ for some $a,b \in \mathbb{Z}$. In terms of the basis $t_1 = t^ks^{-a}$, $s_1 = t^ns^b$, for $\mathbb{Z} \times \mathbb{Z}$ its action on $N$ is given by:
$$
 t_1xt_1^{-1}  = t^ks^{-a}xs^at^{-k} = t^kxt^{-k}  = \psi^k(x)
 $$
 $$
 s_1xs_1^{-1} =t^ns^bxs^{-b}t^{-n} = t^nxt^{-n}= \psi^n(x) = gxg^{-1}
$$

 We see this as also describing a presentation for 
 $$
 N \rtimes_{\psi^k \times  i_g} (\bZ \times \bZ) \cong N \rtimes_{i_{g^{-1}}\circ \psi^k \times id} (\bZ \times \bZ) \cong G_{\psi^k} \times \bZ.
 $$
 Here $i_g$ denotes the inner automorphism given by $g$. The isomorphisms come from the fact that  a semi-direct product depends only on the outer automorphism class of the defining homomorphism.
 
 Thus we have an isomorphism $f: G_\psi \times \mathbb{Z} \to G_{\psi^k} \times \mathbb{Z}$ which is the identity on $N$ and a linear isomorphism on 
 $\mathbb{Z} \times \mathbb{Z}$. Clearly $f(H_\psi \times \mathbb{Z}) = H_{\psi^k} \times \mathbb{Z}$.
 \end{proof}
 
 \begin{lemma} Two pairs $(G,H)$ and $(G^\ast, H^\ast)$ of finitely generated groups with $(G \times \bZ, H\times \bZ) \cong (G^\ast \times \bZ , H^\ast \times \bZ)$ have the same set of finite quotient pairs.
 
 \end{lemma}
 
 \begin{proof} The case $H = H^\ast = \{1\}$ is proved in \cite{Baum}. In order to discuss the general case we need to quickly review some of the basic ideas involved. For any group $G$, and any integer $n >0$,  $G(n)$  will denote the intersection of all normal subgroups of index at most $n$ in $G$. $G(n)$ is a characteristic subgroup of $G$, and providing there are just finitely many subgroups in this intersection (e.g. if $G$ is finitely generated) then $G(n)$ has finite index in $G$, and the sequence $\{G(n): n = 1, 2, \dots\}$ is cofinal among of all finite index normal subgroups of $G$. A characteristic property:
 
 (1) $G/G(n)$ is the largest quotient group (in the sense of mapping onto any other such quotient) of $G$ in which the intersection of all normal subgroups of index $\le n$ is trivial.
 
 We are given an isomorphism
 
$$
f: (G \times \bZ, H \times \bZ) \to (G^\ast \times \bZ, H^\ast  \times \bZ).
$$
Clearly
$$
f((G \times \bZ)(n)) = (G^\ast \times \bZ)(n) 
$$
for all $n$.  Now
$$
(G \times \bZ)(n) = G(n) \times \bZ(n);
$$
where we mean $=$ as subgroups in the same product structure. So
$$
(G \times \bZ)/(G \times \bZ)(n) = G/G(n) \times \bZ/\bZ(n) \text{ and } (H\times \bZ) \cap (G \times \bZ)(n) = H/H\cap G(n) \times \bZ /\bZ(n).
$$
The corresponding observations hold for $G^\ast$ and $H^\ast$ as well.

Noting these identifications we see that $f$ induces an isomorphism 
$$
\bar{f} : G/G(n) \times \bZ/\bZ(n)  \to G^\ast/G^\ast(n) \times \bZ/\bZ(n) \text { with } 
$$
$$
\bar{f}(H/H\cap G(n) \times \bZ/\bZ(n))= H^\ast/H^\ast \cap G^\ast(n)  \times \bZ/\bZ(n).
$$

We are not assuming that $f: G \times \bZ \to G^\ast \times \bZ$ preserves factors. In fact in  the application $f$  will be given by Lemma 2.1 and will not preserve factors. The case for $\bar{f}$  is better. The Remak-Krull-Schmidt Theorem for finite groups (cf Theorem 4.26 of \cite{Rotm})tells us that $G/G(n) \times \bZ/\bZ(n)$ 
and $G^\ast/G^\ast(n) \times \bZ/\bZ(n)$ have unique direct product decompositions into irreducible factors. These decompositions can be obtained by further factoring the given ones. Thus some isomorphism  
$$
g:  G/G(n) \times \bZ/\bZ(n)  \to G^\ast/G^\ast(n) \times \bZ/\bZ(n)
$$
  preserves factors and takes $H/H\cap G(n)$ to $H^\ast/H^\ast(n)\cap G^\ast(n)$.

This establishes for every $n$ an isomorphism of pairs
$$
(G/G(n) , H/H \cap G(n)) \to (G^\ast/G^\ast(n), H^\ast/H^\ast \cap G^\ast(n)).
$$
The lemma now follows from the fact that $\{G(n)\}$ is cofinal. Specifically, let $N$ be a normal subgroup of finite index in G. Then $G(n) \vartriangleleft N$ for $n = [G:N]$. So $(G/N, H/H\cap N)$ is a quotient of $(G/G(n), H/H\cap G(n))$.  Using the isomorphism $(G/G(n), H/H\cap G(n)) \to (G^\ast/G^\ast(n), H^\ast/H^\ast\cap G^\ast(n)$ we see that $(G/N, H/H\cap N)$ is a quotient pair of $(G^\ast, H^\ast)$. Thus every finite quotient pair of $(G,H)$ is a finite quotient pair of $(G^\ast, H^\ast)$. Symmetry of the argument
 completes the proof.
  \end{proof}
  
  \section{Proofs}  
  {\em Proof of Theorem 1.1} 
  $$
  \pi_1(M_\phi) = G_\psi = \pi_1(F) \rtimes _{\psi}  \bZ,  \text{ wiith } \psi = \phi_\ast :\pi_1(F) \to \pi_1(F), 
  $$
  and the analogous statement holds for $\pi_1(M_{\phi^k})$. By Lemma 2.1 
  $$
  \pi_1(M_\phi) \times \bZ \cong \pi_1(M_{\phi^k}) \times \bZ.
  $$
  By \cite{Baum} $\pi_1(M_\phi)$ and $\pi_1(M_{\phi^k})$ have the same finite quotients (hence \cite{DFPR} isomorphic profinite completitions).  \qed
  
  {\em Proof of Theorem 1.2} 
  
  $\pi_1(M_\phi)$ and $\pi_1(M_{\phi^k})$ are as described above. And 
  $$
  i_\ast(\pi_1(\partial M_\phi)) =  \ i_\ast(\pi_1(\partial F)) \rtimes _\phi \bZ =  \ i_\ast(\pi_1(\partial F)) \times \bZ
  $$
  The analogous statement holds for $\pi_1(\partial M_{\phi^k})$.
  
  Thus the proof follows from Lemmas 2.1 and  2.2. \qed
  
  Note that this proof works as well if $M_\phi$ has more than one boundary component-- as long as they all come from $\phi$ - invariant components of $\partial F$. This case easily reduces to the one given, and we have chosen not to complicate its statement. However if we allow $\phi$ to permute components of $\partial F$, then Lemma 2.1 will not apply as stated. Moreover the isomorphism $\pi_1(M_\phi) \times \bZ  \to  \pi_1(M_{\phi^k}) \times \bZ$ given by 2.1 will not take $i_\ast(\pi_1(\partial M_\phi)) \times \bZ $ to $\i_\ast(\pi_1(\partial M_{\phi^k}) \times \bZ$; so Lemma 2.2 cannot be applied, and the status of the theorem is uncertain in this case.
  
  \section{Surface bundles vs Seifert fibrations} 
 The compact oriented Seifert fibered spaces with oriented base are determined, up to orientation preserving, fiber preserving homeomorphism by the ``classical'' Seifert invariants:
  
  $g = $ {\em  genus } and $s$ = {\em number of boundary components} of  the base $B$,
  
  {\em the fiber invariants} $\beta_i/\alpha_i; i = 1, \dots, m $ normalized with $0 < \beta_i < \alpha_i$,
  
  {\em and, in the case $s = 0$ (i.e. $M$ closed), the obstruction } $b$,
  
\noindent which describe the manifold as Dehn filling on $B^\ast\times S^1$ where $B^\ast$ is a compact,  oriented surface of genus $g$ with $\partial B^\ast$ as described below.
  
  For the case $s=0$,  $B^\ast$ will have  $m+1$ boundary curves  
  $\{x_0,x_1, \dots , x_m\}$  oriented by $B^\ast$, $t$ will be  an oriented generator for $S^1$, and the filling relations will be $\{x_i^{\alpha_i}t^{\beta_i} =1; i = 1,\dots,m\}$ and $x_0t^b=1$.
  
  For the case $s\ne 0$, $B^\ast$ will have $ m+s$ boundary components, and the Dehn filling  will be done, as above,  on the first $m$ of them.
  
  The oppositely oriented manifold can be obtained by reversing the fiber orientation. Its invariants are: same $g$ and $s$,  fiber invariants $(\alpha_i - \beta_i)/\alpha_i; i = 1, \dots, m$, and  (when $s=0$) obstruction $-b-m$.
  
  The {\em Rational Euler number}   defined (when $s=0$)  by:
  $$
  e = -(b + \sum\limits_{i=1}^m \beta_i/\alpha_i),
  $$
  is in many ways (c.f. \cite{NeRa}) a more natural invariant than $b$ -- as we will see below.
  
  The following result is well known (c.f. Theorem 5.4 of \cite{Scot} for the closed case and  \cite{Ja-1}, Ch VI for  the bounded case).  In the following subsection we give a constructive proof which will be useful in translating examples from one representation to the other
  
    \begin{theorem}  If $M$ is a compact, oriented Seifert fibered space with orientable base then $M$ is also a surface bundle over $S^1$ with periodic monodromy if and only if either $\partial M \ne \emptyset$  or $M$ is closed and  $e(M) = 0$.
  \end{theorem}
  
  As discussed in \cite{Scot} the geometry of $M$ as above will be one of $S^2 \times \bR, \mathbb{E}^3,  \bH^2 \times \bR$.
  
  Strictly speaking  we are abusing notation as these  invariants  depend on the Seifert structure -- not just the underlying manifold. However the   only manifolds to which this Theorem applies  which have non-unique Seifert fiberings (with orientable base) are  $S^2 \times S^1$; which can be represented with $g = 0, m = 2, b=-1$, and fiber invariants $\beta/\alpha,( \alpha -\beta)/\alpha$ for any $\alpha > 1$, $B^2 \times S^1$, and $S^1 \times S^1 \times I$. 
  
  Moreover  manifolds in this class are determined by their fundamental groups (if closed) or peripheral fundamental group systems (if bounded). However:

 \begin{lemma} The oriented Seifert fiber space $M$ with invariants : $g$, $s $, and $\{\beta_i/\alpha_i; i = 1, \dots, m\}$ with $s >0$ has fundamental group $\pi_1(M)$ independent of $\beta_1, \dots, \beta_m$.
 \end{lemma}
 
 \begin{proof} Choose a component $J$ of $\partial B$ and arcs $(a_i,\partial a_i) \subset (B, J); i = 1, \dots, m$ which cobound a collection $\{D_i; i=1, \dots,m\}$ of disjoint disks in $B$ with arcs in $J$ such  that each $D_i$ contains exactly one singular point. Let $\eta: M \to B$  be the bundle projection. Put $B^\ast = Cl(B - \cup D_)$. Then 
 $$
 M^\ast =\eta^{-1}(D^\ast)  \cong B^\ast \times S^1.
 $$
 with the curves $z \times S^1$ being regular fibers. Each $\eta^{-1}(D_i)$ is a solid torus with $\eta^{-1}(a_i)$ an annulus which is a union of regular fibers which wrap $\alpha_i$ times the singular (center) fiber.
 
 So $\pi_1(M)$ is obtained from $\pi_1(M^\ast) = \pi_1(B^\ast) \times <t: \,\, >$ by adding generators $\lambda_i; $ 
 and relations $\lambda_i^{\alpha_i} = t; i = 1, \dots,m$.
 \end{proof}
  
 \subsection*{ The relation between Seifert fibered  and surface bundle structures} 
  Let  $M = M_\phi$ where $\phi:F \to F$ is periodic of order $n$.  Here $F$ will be a compact, oriented surface with or without boundary. The fibers in the Seifert structure on $M_\phi$ are obtained from the intervals $x \times I$: Begin at $(x,0)$ proceed along $x \times I$ to $(x,1)= (\phi(x),0)$, and so on until the curve closed up. The regular fibers come from the points in $F$ whose orbit under the action of the cyclic group $C$ generated by $\phi$ consists of $n$ points.  The singular fibers correspond to the smaller  orbits -- i.e. points with non-trivial stabilizer.
  
 As $C$ is abelian, all points in each orbit have the same stabilizer. So let $z \in F$ be a point. whose stabilizer has order $p >1$. Then $\phi^{n/p}$  generates $stab(z)$ and is the smallest non-trivial power of $\phi$ fixing $z$. $\phi^{n/p}$ will rotate a small invariant disk neighborhood  $D$ of $z$ in $F$ by $2\pi q/p$ for some $q$  prime to $p$. The union of all fibers through $D$ is a $(p,q)$-fibered solid torus. Thus:
 
 {\em (i) The fiber invarient $\beta/\alpha$ associated to the singular fiber through $z$ has $\alpha = p =o(stab(z))$ and with and $q\beta \equiv  1 \mod p$} and $0< \beta <\alpha$.
 
 {\em (ii) $M$ has an $n$-sheeted covering by the product $\tilde M = F \times I$.  By naturality of $e$ (\cite{NeRa}), $0 = e(\tilde M) = n e(M)$; so $e(M) =0$.}
  
  In order to describe the surface bundle structure from the Seifert fibering we need some comments about the associated group $\Gamma = \pi_1(M)/<t>$ where $t$ is the (central)  element represented by a regular fiber. $\Gamma$  has the presentation:
  $$
  < x_1, \dots, x_m, a_1,b_1, \dots, a_g,b_g, y_1, \dots, y_s: 
  $$
  $$
   x_1 \dots x_m[a_1b_1]\dots [a_g,b_g]y_1\dots y_s =1, x_1^{\alpha_1}= \dots=x_m^{\alpha_m} = 1>
  $$
$\Gamma$ is the {\em orbifold} fundamental group of $B$ which describes $B$ as a quotient of $S^2, \bR^2, \text {or } \bH^2$ according as the orbifold Euler characteristic:
$$
\chi^{orb}(B) = 2-2g-s+ \sum\limits_{i=1}^m (1/\alpha_i - 1)
$$
of $B$ is positive, zero, or negative. $\Gamma$ is called a {\em Fuchsian Group} when $\chi^{orb}(B) < 0$.

 $\Gamma$ is said to be of {\em odd type} if $s=0$, $\lambda = lcm\{\alpha_1, \dots, \alpha_m\} $ is even, and $\lambda/\alpha_i$ is odd for an odd number of $i \in \{1, \dots,m\}$. Otherwise $\Gamma$ is said to be of {\em even type}. The following is proved in \cite{EdEK}.
 
 {\em (iii) If $\chi^{orb}(B) \le 0$ then $\Gamma$ has a torsion free subgroup of finite index $k$ if and only if $2^\epsilon\lambda$ divides $k$ where $\epsilon =0$ if $\Gamma$ has even type and $\epsilon = 1$ if $\Gamma $ has odd type, and $\lambda = lcm\{\alpha_1, \dots, \alpha_m\}$}.
 
 The following may be of independent interest.
 \begin{lemma} 
 A closed, oriented Seifert fiber space $M$ with $e(M) = 0$ must have $\Gamma = \pi_1(M)/<t>$ be of even type.
 \end{lemma}
\begin{proof}
Suppose $e(M)=0$. Then $\sum_i \beta_i/\alpha_i  = -b $ is an integer. If we have odd type, then $\lambda$ is even. Moreover
$$
\sum_i\beta_i/\alpha_i = \{\sum_i\beta_i \lambda/\alpha_i\}/\lambda.
$$
By definition $\lambda/\alpha_i$ is odd for an odd number of $i$. The corresponding $\alpha_i $ must be even, and as $(\alpha_i,\beta_i)=1$ the corresponding $\beta_i$ must be odd. So the numerator in this fraction is odd. The denominator is even so the fraction is not an integer.
\end{proof}
 
 Now let $M$ be a compact, oriented  Seifert fibered space over an oriented  base $B$. If $\partial M = \emptyset$ assume $e(M) = 0$. We want to describe  $M = M_\phi$ for some periodic homeomorphism of some surface $F$. Unfortunately most surface bundles over $S^1$ do not have unique surface bundle structures. Uniqueness only holds if $\beta_1(M) = 1$ -- so that there is a unique homomorphism $\pi_1(M) \to \bZ$. In fact for the closed oriented surface $S_g$ of genus $g$, and for any $n \ge 1$,  $S_g \times S^1$ is a surface bundle over $S^1$ with fiber $S_{\tilde g}$ with $\tilde g = n(g-1)+1$  \cite{Ja-1}. The ``best'' surface bundle structure on $M$ is given by
  
\begin{lemma} Let $M$ be a compact, oriented  Seifert fibered space over an oriented  base $B$. If $\partial M = \emptyset$ assume $e(M) = 0$. Then $M$  is a surface bundle over $S^1$ with periodic monodromy $\phi$  of  order   $  = \lambda = $  the minimal index of a torsion free subgroup   of  $\pi_1(M)/<t>$.
\end{lemma}
 
 \begin{proof}

  The four closed  Euclidean manifolds 
 $$(-1; 1/2,1/4,1/4), (-1;1/2, /1/3,1/6), (-1;1/3,1/3,1/3), \text {and } (-2,1/2,1/2,1/2,1/2)
 $$ are known to be torus bundles over $S^1$ with monodromies of orders $4,6,3,2$ respectively. So we assume $\chi^{orb}(B) <0$.
 
  By Theorem 4.1, comment (iii),  and Lemma 4.3 this minimal index is $\lambda = lcm\{\alpha_1, \dots, \alpha_m\}$.  We will construct the surface bundle structure on $M$ in much the same way the bundle structures on $S_g \times S^1$ were constructed in \cite{Ja-1}. We concentrate on the case $\partial M = \emptyset$.  Recall that $M$ can be constructed by Dehn filling on $ M_0 =B^\ast \times S^1$ where  $B^\ast$ is a surface with $ m+1$ boundary components $x_0, x_1, \dots ,x_m$. Let $T_i  = \ x_i \times S^1$.
  
  We start with the surface $F_0 = \bigcup \limits_{i=0}^{\lambda -1} B^\ast \times \{i/\lambda\}$ consisting of $\lambda $ parallel copies of $B^\ast$ (regarding $S^1 = [0,1]/ 0 \sim 1$). Get disjoint arcs $a_1, \dots, a_m$ in $B^\ast$ ; with $a_i$ joining a point of $x_i$ to a point of $x_0$. The annulus $A_i = a_i \times S^1$ Meets $F_0$ in $\lambda$ parallel arcs in $A_i$.
  
  For each $i = 1, \dots ,m$ we split  $M_0$ along $A_i$, rotate the negative side by the fraction
  $$
 ((\lambda/\alpha_i)\beta_i)/\lambda = \beta_i/\alpha_i
  $$
   of a full revolution and re-glue. This takes $F_0$ to an oriented surface $F_1$. For $i \ge 1$ the  homology class of $\partial F_0 \cap T_i \in H_1(T_i)$ is $\lambda [x_i]$. The homology class of $\partial F_1 \cap T_i$ is
  $$
  \lambda [x_i] + (\lambda/\alpha_i)\beta_i [t] = (\lambda/\alpha_i)(\alpha_i[x_i] +\beta_i[t]).
  $$
  As $\alpha_i$ and $\beta_i$ are relatively prime integers, this means that $F_1 \cap T_i$ consists of $\lambda/\alpha_i $ parallel copies of simple closed curves representing the class $\alpha_i[x_i] + \beta_i[t]$.
  
  Now $\partial F_0 \cap T_0$ represents the homology class $-\lambda [x_0]$; so $\partial F_1 \cap T_0$  represents the class
  $$
  -\lambda[x_0]  + (\sum\limits (\lambda/\alpha_i)\beta_i[t]) = \lambda( -[x_0] + \sum\limits_{i=1}^m \beta_i/\alpha_i) = \lambda(-[x_0] -b[t]).
  $$
 The equalities in the line above take place with integer coefficients and use the assumption $e(m) = 0$.  This shows that $F_1 \cap T_0$ consists of $\lambda$ copies of a simple closed curve representing the class $[x_0]+b [t].$
 
 The boundary curves of $F_1$ on $T_i$ are the same as the Dehn filling parameters which produce $M$ from $M_0$. Thus we can cap off the boundary components of $F_1$ in $M-M_0$ to obtain a closed surface $F \subset M$ which meets every regular fiber in $\lambda$ points -- all with intersection number  $+1$.  If we split $M$ along $F$ we get an oriented $I$-bundle over the oriented  $F$. This must be the product bundle $F \times I$. Re-gluing makes $M$ an $F$ bundle over $S^1$ whose monodromy $\phi: F \to F$ takes $x \in F$ to the next point of $F$ on the fiber through $x$. Clearly $order(\phi) = \lambda$.
 
 Note that it automatically follows that $F$ is connected. Otherwise a component will produce a torsion free subgroup of $\pi_1(M)/<t>$ of index  $ < \lambda$.
 
 The case $\partial M \ne \emptyset$  proceeds just as above with $x_0$ a component of $\partial B^\ast$ with $x_0 \times S^1$ a component of $\partial M$. There is no need to cap off the components of $\partial F_1$ in $x_0 \times S^1$.
 \end{proof}
 
 \begin{cor} For the  Fuchsian group $\Gamma \cong \pi_1(M)/<t>$, as above, with $e(M) =0$,  the torsion free subgroup of minimal index in $\Gamma$ is a normal subgroup. 
 
 \end{cor}
 \begin{proof} There is a $\lambda$ sheeted regular covering $\rho :M_{\phi^\lambda} \to M_\phi = M$. But $M_{\phi^\lambda} = F \times S^1$ ($\rho$-equivariently), and $\rho_\ast(\pi_1(M_{\phi^\lambda}))$ is the pullback of this minimal index torsion free subgroup of $\Gamma$.
 \end{proof} 
 
 \section{Examples}
 
 We now fix a compact, oriented Seifert fibered space  over an oriented base $B$ with  normalized fiber invariants
 
 $$
 \{\beta_1/\alpha_1, \dots, \beta_m/\alpha_m\}
 $$
 We assume $M \ne S^2 \times S^1, B^2 \times S^1, S^1 \times S^1 \times I$.
 If $\partial M = \emptyset$ assume $e(M) =0$;
 so $M = M_\phi$ for some periodic $\phi :F \to F$ of order $\lambda = lcm\{\alpha_1, \dots ,\alpha_m\}$. Take $k$ relatively prime to $\lambda$.
  
 \begin{prop} $M_{\phi^k}$ has fiber invariants $\{\beta_1^\ast /\alpha_1, \dots, \beta_m^\ast/\alpha_m\}$ where $0 < \beta_i^\ast < \alpha_i$ and $k \beta_i^\ast \equiv  \beta_i \mod \alpha_i$ for $i = 1, \dots, m$.
 \end{prop}
 
 \begin{proof} By (i) above $\alpha_i$ is the order of the stabilizer of any point $z \in F$ which lies in the $i^{th}$ singular fiber, and $q \beta_i \equiv 1 \mod \alpha_i$ where $\phi^{\lambda/\alpha_i}$ rotates an invariant neighborhood of $z$ by $2\pi q/\alpha_i$.
 
 Then $\lambda/\alpha_i$ is the smallest power of $\phi^k$ fixing $z$, and $(\phi^k)^{\lambda/\alpha_i}$ rotates this disk by $2\pi k q/\alpha_i$.
 \end{proof}
 
 \begin{cor}  $M_\phi$ and $M_{\phi^k}$ are  homeomorphic if and only 
if $\{\beta_1^\ast/\alpha_1, \dots,\beta_m^\ast/\alpha_m\}$ is the same as one of $\{\beta_1/\alpha_1, \dots, \beta_m/\alpha_m\}$ or $\{(\alpha_1 - \beta_1)/\alpha_1, \dots, (\alpha_m - \beta_m)/\alpha_m\}$.
 \end{cor}
 
 \begin{proof} $M_\phi$ and $M_{\phi^k}$ have the same base. If $\partial M \ne \emptyset$ then $e(M_{\phi^k}) = 0$. We are dealing with the class of Seifert fibered spaces which are determined by their Seifert invariants.
  \end{proof}
 
 Combining this corollary with Theorems 1.1 and 1.2 gives the desired examples
 
 .
 
 Note: For  those (closed) $M$ with the geometry of  $\bE^3$ we have $\lambda = 2,3,4, \text{or } 6$ and the only possible values for $k$ are $k = \pm 1$, and so $M_\phi$ and $M_{\phi^k}$ are necessarily homeomorphic.
 
 Many, but not all, manifolds with the geometry $\bH^2 \times \bR$ will provide  examples with $M_\phi$ not homeomorphic to $M_{\phi^k}$. We give a couple of sample theorems. We are now assuming that $\chi^{orb}(B) <0$.
 
 \begin{prop} If $\alpha_1 \ne 2,3,4,6$ and is distinct from $\alpha_i$ for $i >1$ then there is some $k$ prime to $\lambda$ with $\pi_1(M_\phi)$ not isomorphic to  $\pi_1(M_{\phi^k})$.
 \end{prop}
 
 \begin{proof} It suffices to find $k$ prime to $\lambda$ with $k \not\equiv \pm 1 \mod \alpha_1$. If $\lambda$ and $\alpha_1$ have the same prime divisors, then by the restrictions on $\alpha_1$, we can find such a $k$ prime to $\alpha_1$ -- hence to $\lambda$.
 
 In general write $\lambda = \lambda_1 \lambda_2$ where $\alpha_1$  and $\lambda_1$ have the same prime divisors and $(\lambda_1, \lambda_2) = 1$. By the first case get $k_1$ prime to $\lambda_1$ with $k_1  \not\equiv  \pm1 \mod \alpha_1$.  By the Chinese remainder theorem get $k$ with $k \equiv k_1 \mod \lambda_1$ and $k \equiv  1 \mod \lambda_2$. This $k$ satisfies our requirements.
 \end{proof}
 
 \begin{prop} Let $p$ be a prime with $p \ge 7$, and let $\{\beta_1, \dots, \beta_m\}$  represent the multiplicative group $H$ of order $m = (p-1)/2$ in $\bZ/(p)$.  Let $M$ be the Seifert fibered space with base $S^2$, fiber invariants $\{\beta_1/p, \dots, \beta_m/p\}$ and $b= -(\sum\beta_i)/p$ --necessarily an integer; so $e(M) = 0$ and $M = M_\phi $ with $order(\phi) = p$. Then for any $k$ prime to $p$ $M_\phi$ and $M_{\phi^k}$ are homeomorphic.
 \end{prop}
 
 \begin{proof} Multiplication by $k$ either permutes the $\beta_i$'s or takes them to their negatives (modulo $p$) which represent the nontrivial coset of $H$.  Note $p\ge 7$ is necessary to make $\chi^{orb}(B) < 0$.
 \end{proof}
 
 The case of manifolds with boundary is nicely illustrated by the Seifert fibered spaces over $B^2$ with two singular fibers. Let  $M$ be such and let the fiber invariants be $\beta_1/\alpha_1, \beta_2/\alpha_2$. $M$ is called a ``Lens space torus knot complement'' as $M$ is homeomorphic to the complement of a regular neighborhood of a knot on a genus one Heegaard surface for some Lens space (we count $S^2 \times S^1$ and $S^3$ as Lens spaces). 
 
 These Lens spaces are precisely those with base $S^2$, fiber invariants  $\beta_1/\alpha_1, \beta_2/\alpha_2$, and obstruction $b$ -- for any $b \in \bZ$. This is $L_{p,q}$ where:
 $$
 p = |\alpha_1 \beta_2 + \alpha_2 \beta_1 + b \alpha_1 \alpha_2| \text{ and }  q\equiv -(\gamma_1\alpha_2 +\delta_1\alpha_2 b) \mod p
 $$
 where $\alpha_1 \gamma_1 - \beta_1 \delta_1 = 1$.
 
 Now fix $\alpha_1$ and $\alpha_2$ and  denote by ${\mathcal C}(\alpha_1, \alpha_2)$  the set of Seifert fibered spaces over $S^2$ with fiber invariants $\{\beta_1/\alpha_1, \beta_2/\alpha_2\}$ for some choice of $\beta_1, \beta_2$. We assume $(\alpha_1, \alpha_2) = 1$ -- a necessary condition that ${\mathcal C}(\alpha_1,\alpha_2)$ contain a torus knot space in $S^3$. We count the elements of  ${\mathcal C}(\alpha_1,\alpha_2)$ two different ways.
 
 First there are $\varphi(\alpha_1)$ (Euler function) possible choices for $\beta_1$ and $\varphi(\alpha_2)$ choices for $\beta_2$. So there are $\varphi(\alpha_1) \varphi(\alpha_2) = \varphi(\alpha_1 \alpha_2)$  choices   of fiber invariants. These invariants distinguish the oriented manifolds. Allowing orientation reversing homeomorphisms cuts this number precisely in half: there are $\varphi(\alpha_1 \alpha_2)/2$ distinct
 manifolds in ${\mathcal C}(\alpha_1,\alpha_2)$.
 
 Next we note that for any $M \in {\mathcal C}(\alpha_1,\alpha_2)$,  $M = M_\phi$ where $\phi: F \to F$ has period $\lambda = \alpha_1 \alpha_2$ and $F $ has connected boundary. There are $\varphi(\lambda)$ different oriented manifolds $M_{\phi^k}$. But as $M_{\phi^k}$ is orientation reversing homeomorphic to $M_{\phi^j}$ if and only if $k \equiv -j \mod \lambda$, we have $\varphi(\lambda)/2$ distinct $M_{\phi^k}$'s. This establishes:
 
 \begin{prop} For $(\alpha_1, \alpha_2) = 1$
 
 (i) ${\mathcal C}(\alpha_1,\alpha_2)$ has $\varphi(\alpha_1 \alpha_2)/2$ elements, all of which have the same fundamental group,
 
 (ii) Any pair of distinct elements of ${\mathcal C}(\alpha_1,\alpha_2)$ have the same sets of finite peripheral pair quotients though their peripheral fundamental group systems are (necessarily) distinct,
 
 (iii) Every element of ${\mathcal C}(\alpha_1,\alpha_2)$ is some   cyclic covering space of every other element.
 
 \end{prop}
 \begin{proof} (i) follows from Lemma 4.2.
 
 (ii) follows from Theorem 1.2.
 
 For (iii) we note that $M_{\phi^k}$ is a $k$-sheeted cyclic cover of $M_\phi$, and for $j k \equiv 1 \mod \alpha_1 \alpha_2$, $M_\phi = M_{(\phi^k)^ j}$ is a $j$-sheeted cyclic cover of $M_{\phi^k}$.
  \end{proof}
 
 It may be of interest to note that the elements of ${\mathcal C}(\alpha_1,\alpha_2)$  can be distinguished by the ``minimal'' Lens space in which they embed and that this will be the one of minimal homology which contains the given $M$ as a  torus knot complement as described above.


 \end{document}